\documentclass[11pt]{article}

\usepackage{amsmath}
\usepackage{amsthm}
\usepackage{amssymb}
\usepackage{mathrsfs}
\usepackage{graphicx}
\def\div{{\,\rm div \,}}
\def\cof{{\,\rm cof \,}}
\def\loc{{\,\rm loc \,}}

\def\LL{{\,\rm L \,}}

\def\CC{{\,\rm C\,}}
\def\WW{{\,\rm W\,}}

\def\qoq{{\quad\mbox{on}\quad}}
\def\qaq{{\quad\mbox{at}\quad}}
\def\qasq{{\quad\mbox{as}\quad}}

\def\id{{\,\rm id \,}}

\def\sym{{\,\rm sym \,}}

\def\dist{{\,\rm dist \,}}

\def\SO{{\,\rm SO \,}}

\def\supp{{\,\rm supp \,}}

\def\+M{{\,\rm M^{n\times n}_+ \,}}
\def\tr{{\,\rm tr \,}}

\def\qfq{{\quad\mbox{for}\quad}}
\def\qflq{{\quad\mbox{for all}\quad}}

\def\supp{{\,\rm supp \,}}

\def\lam{\lambda}

\def\t#1{{\tilde#1}}

\def\vare{\varepsilon}

\def\E{{\cal E}}

\def\X{{\cal X}}

\def\E{{\cal E}}

\def\H{{\cal H}}

\newfont{\Blackboard}{msbm10 scaled 1200}

\newfont{\roma}{cmr10 scaled 1200}

\def\<{{\langle}}
\def\>{{\rangle}}

\def\var{\varphi }
\def\si{\sigma}

\def\a{\alpha}
\def\b{\beta}
\def\om{\omega}
\def\Om{\Omega}

\newtheorem{thm}{{}\hskip\parindent Theorem}[section]
\newtheorem{lem}{{}\hskip\parindent Lemma}[section]
\newtheorem{pro}{{}\hskip\parindent Proposition}[section]

\def\pl{\partial}
\def\rw{\rightarrow}

\def\be{\begin{equation}}
\def\ee{\end{equation}}
\def\beq{\arraycolsep=1.5pt\begin{eqnarray}}
\def\eeq{\end{eqnarray}}

\def\R{I\!\!R}

\def\n{\vec{n}}
\title{On Exponents of Thickness in
Geometry Rigidity Inequality for Shells}
\date{}
\author{Liang-Biao Chen and Peng-Fei  Yao\\[0.2cm]
School of Mathematics and Statistics\\\nonumber
Key Laboratory of Complex Systems and Data Science of Ministry of
Education\\\nonumber
Shanxi University, Taiyuan, 030006,
China\\\nonumber
e-mail: pfyao@iss.ac.cn}

\begin{document}

\maketitle
 \footnote{This work is supported by the National
Science Foundation of China, grants no. 12071463, and by the special fund for Science and Technology Innovation Teams of Shanxi Province, grants no. 202204051002015. }

\begin{quote}
\begin{small}
{\bf Abstract} \,\,\,We study exponents of thickness in Frieseck-James-M\"uller's inequalities for shells. We derive the following results: (a) the exponent of thickness $\mu(S)\leq15/8$ if the middle surface $S$ is parabolic; (b) the exponent of thickness $\mu(S)\leq11/6$ if the middle surface $S$ is a minimal surface with negative curvature; (c) the exponent of thickness $\mu(S)\leq11/6$ if the middle surface $S$ is a ruled surface with negative curvature.
The  exponents of thickness in Frieseck-James-M\"uller's inequalities for thin shells
represent the relationship between rigidity and  thickness $h$ of a shell when the large
deformations take place, i. e., the rigidity of the shell related to the thickness $h$ is
$$Ch^{\mu(S)}.$$Thus the above results of $\mu(S)<2$ show that those shells are strictly more rigid than plates since $\mu(S)=2$ for plates. Moreover, we present another result which shows that when $\mu(S)<2,$
any $\WW^{2,2}$ isometry of the middle surface is rigid.
\\[3mm]
\end{small}
\end{quote}

\section{Introduction and Main Results}
\def\theequation{1.\arabic{equation}} The geometry rigidity inequality was established in \cite{FrJaMu} in 2002 by G. Friesecke,
R. James and S. Muller. This inequality plays the pivotal role in solid mechanics, nonlinear elasticity and differential geometry, and it is based on the intuition of quantifying the well-established fact that a deformation that is locally a rigid motion, must be a global rigid motion on each connected component of its domain.  In fact, Frieseck-James-M\"uller's inequalities are the cornerstone of rigorous derivations
of two dimensional plate and shell theories from three-dimensional models in the framework of
nonlinear elasticity theory, see \cite{CY, CY1, FrJaMu,FrJaMu1,FJMM,HoLePa,LeMoPa,LeMoPa1,Yao2017} and many others.

In their celebrated work, Friesecke et al. \cite{FrJaMu,FrJaMu1} derived that
the thickness exponent in the geometric rigidity inequality is $2$ for plates, which gave rise to derivation of a hierarchy of plate theories for different scaling regimes of the elastic energy depending on
the thickness $h$ of the plate \cite{FrJaMu1}. This type of theory has been derived by Gamma-convergence and relies on $L^p$-compactness arguments and of course the geometry rigidity inequality, which plays the main role.
Similarly, the exponents of thickness in the geometry rigidity inequality for shells will improve the Gamma-convergence models of shells if the exponents are strictly smaller than $2.$

 Let $S\subset\R^3$ be an oriented  surface with a normal field $n.$  A shell with middle surface $S$ and thickness $2h>0$ is defined by
$$S_h=\{\,x+tn(x)\in\R^3\,|\,x\in S,\,\,|t|<h\,\}.$$ Let $\mu>0$ be a constant such  that there exists $C>0,$ independent of $h>0,$ such that for every $u\in\WW^{1,2}(S_h,\R^3)$ there is  a constant rotation $R\in\SO(3)$ satisfying
\be \|\nabla u-R\|^2_{\LL^2(S_h)}\leq\frac{C}{h^\mu}\|\dist(\nabla u,\SO(3))\|^2_{\LL^2(S_h)}.\label{1.1}\ee
If a displacement $u$ satisfies the zero boundary conditions, the inequality above takes a form of (\ref{1.2}) below. Set
$$\H_0^1(S_h,\R^3)=\{\,u\in\WW^{1,2}(S_h,\R^3)\,|\,u|_{(\pl S)_h}=0\,\}$$ with the norm of $\WW^{1,2}(S_h,\R^3),$ where
$$(\pl S)_h=\{\,x+tn(x)\,|\,x\in\pl S,\,\,|t|<h\,\}.$$ There exists $C>0$  such that
\be\|\nabla u\|^2_{\LL^2(S_h)}\leq\frac{C}{h^\mu}\|\dist(\nabla u+I,\SO(3))\|^2_{\LL^2(S_h)}       \label{1.2}\ee for all $u\in\H_0^1(S_h,\R^3)$ and all $h>0,$ where $I$ is the unit matrix in $\R^{3\times3}.$

Set
$$\mu(S)=\inf\{\,\mu\,|\,\mu>0\,\,\mbox{is such that (\ref{1.1}) or (\ref{1.2}) holds}\,\}. $$ $\mu(S)$ is said  to be the exponent of thickness in the geometry rigidity inequality of shells.
Friesecke, R. James, and  S. Muller established in \cite{FrJaMu1} that, if  $S_h$ is a plate, then
$$\mu(S)=2.$$
M. Lewicka, M. G. Mora, and M. R. Pakzad proved in \cite{LeMoPa} for a general shell
$$\mu(S)\leq2.$$ We here present a proof that (\ref{1.2}) holds true for $\mu=2,$ see Theorem \ref{t4.1} later. Moreover, some lower bounds were given by \cite{Yao2021}: $\mu(S)\geq1$ if $S$ is elliptic, $\mu(S)\geq4/3$ if $S$ is hyperbolic, and $\mu(S)\geq3/2$ if $S$ is parabolic.

It is known that the exponents of thickness in Frieseck-James-M\"uller's inequalities for thin shells
represent the relationship between rigidity and thickness $h$ of a shell when the large
deformations take place. Thus the rigidity of the shell about the thickness is
$$Ch^{\mu(S)}$$ in large vibrations. In general $\mu(S)$ is subject to the curvature of the middle surface.
Then different values of $\mu(S)$ reflect the different geometrical rigidities of the shells: the smaller $\mu(S)$ is,  more rigid the shell. The similar problem
has been well understood  in small vibrations, see \cite{GHa1,GHa, Har} and \cite{CY2, Yao2020,Yao2021l}, where instead of the geometry rigidity inequality the Korn inequality is concerned.

In contrast, the case $\mu(S)<2$ is rather critical. For one thing, the case $\mu(S)<2$  shows that the shell is really more rigid than a plate where $\mu(S)=2.$   For another thing, we shall show that the case $\mu(S)<2$ implies that the middle surface $S$ is at least $\WW^{2,2}$ rigid: Any $\WW^{2,2}$ deformation of $S$ with the boundary fixed is an identical map. The problem of rigidity of surfaces is a classic one in differential geometry  and  we refer to \cite{CDS,DeL} for a short concise review.\\

We state our main results as follows.

Let $M\subset\R^3$ be a $\CC^2$ surface with  negative Gaussian curvature. A Lipschitz region $S\subset M$ is said to be {\it regular} if there exists a  non-characteristic region $\t{S}$ in $M$ such that
$$S\subset\subset\t{S}.$$ The non-characteristic-ness  of a region is a technical assumption on the boundary of the region for the regularity of solutions to the linear strain tensor. There are several types of non-characteristic regions in \cite{Yao2017,Yao2018}.

Surface $S$ is said to be parabolic if the following is satisfied:
$$\kappa=0,\quad \nabla n\not=0\qflq p\in\bar{S},$$ where $\kappa$ is the Gaussian curvature.

\begin{thm}\label{t1.1}Let $S\subset\R^3$ be bounded $\CC^2$ parabolic with a Lipschitz boundary. Then, for any $1/8>\theta>0$ small there exist constants $C_\theta>0$ and $h_\theta>0$ such that
$$\|\nabla u\|^2_{\LL^2(S_h,\R^{3\times3}))}\leq C_\theta h^{-15/8-\theta}\|\dist(\nabla u+I,\SO(3))\|^2_{\LL^2(S_h)},$$ for all
$u\in\H_0^1(S_h,\R^3)$ and all $0<h<h_\theta.$
\end{thm}

It is known that if surface $S$ is parabolic, then it can be parameterized as a ruled surface. Next we consider  a case where the middle surface is given by a ruled $\CC^2$ surface but has negative curvature. Let
\be S=\{\,\om(s)\,|\,s=(s_1,s_2)\in(a_1,a_2)\times(b_1,b_2)\,\},\label{om}\ee where $\om(s)=c(s_1)+s_2\delta(s_1)$ and $c,$ $\delta$ are two $\CC^2$ curves in $\R^3.$ Then
$$\pl_1=c'(s_1)+s_2\delta'(s_1),\quad \pl_2=\delta(s_1).$$
Moreover, suppose that $\pl_1$ and $\pl_2$ are linearly independent on $\bar{\t{S}}.$

\begin{thm}\label{t1.2}Let $S\subset\R^3$ be a bounded $\CC^2$  surface with  negative Gaussian curvature. Suppose there  is a non-characteristic  region $\t{S}$ such that
$$S\subset\subset\t{S}.$$ Further assume that one of the following assumptions $i)$ or $ii)$ holds:

$i)$\,\,\,$S$ is a minimal surface, or

$ii)$\,\,\,$S$ is a ruled surface, given in $(\ref{om}).$

Then, for any $1/6>\theta>0$ small there exist constants $C_\theta>0$ and $h_\theta>0$ such that
$$\|\nabla u\|^2_{\LL^2(S_h,\R^{3\times3}))}\leq C_\theta h^{-11/6-\theta}\|\dist(\nabla u+I,\SO(3))\|^2_{\LL^2(S_h)},$$ for all
$u\in\H_0^1(S_h,\R^3)$ and all $0<h<h_\theta.$
\end{thm}

\begin{thm}\label{t1.4} Let an oriented surface $S\subset\R^3$ be a bounded $\CC^2$ Lipschitz domain. Then

$(i)$\,\,\, If there is $0<\mu<2$ such that $(\ref{1.1})$ holds true, then any isometry $w\in\WW^{2,2}(S,\R^3)$ is rigid, i.e., there exist $R\in\SO(3)$ and $a\in\R^3$ satisfying
$$w(x)=Rx+a\qfq x\in S.$$

$(ii)$\,\,\,If there is $0<\mu<2$ such that $(\ref{1.2})$ holds true, then any isometry $w\in\WW^{2,2}(S,\R^3)$ satisfying
\be w(x)=x,\quad \nabla_{\tau}w(x)=\tau\qfq x\in\pl S\label{1.3}\ee
is the identical map, where $\tau$ is the outside normal along $\pl S.$
\end{thm}

Form Theorems \ref{t1.1} -- \ref{t1.4}, we have the following immediately.
\begin{thm}Suppose that one group of the assumptions in Theorems $\ref{t1.1}$ or $\ref{t1.2}$  holds true. Then
any isometry $w\in\WW^{2,2}(S,\R^3)$ satisfying
$$ w(x)=x,\quad \nabla_{\tau}w(x)=\tau\qfq x\in\pl S$$
is the identical map.
\end{thm}

The remaining part of the paper is organized as follows. We make some preparations in Sections $2$ and $3.$ The proofs of Theorems \ref{t1.1}-\ref{t1.4} are given in Section 4.

\setcounter{equation}{0}
\section{Truncation Theorems}
\def\theequation{2.\arabic{equation}}
Moving towards proofs of the results in Section 1, we present here a useful Lusin-type estimate, which allows us to replace a given Sobolev function by its Lipschitz "truncation". The two functions agree on a large set, whose complement has its measure controlled (inversely proportionally) in terms of the requested Lipschitz constant. Such results  were  used in \cite{FrJaMu} for the proof of  the geometry rigidity inequality.
The novel point here is that the constants $C$ in (i)-(iii) below can be chosen to be independent of $h>0.$

\begin{thm}\label{t2.1} For every $u\in H^1(S_h)$ and every $\lam>0$ there exists  $v\in \WW^{1,\infty}(S_h)$ satisfying:

$(i)$\,\,\,$\|\nabla v\|_{\LL^\infty(S_h)}\leq C\lam,$

$(ii)$\,\,\,$|\{\,x\in S_h\,|\,u(x)\not=v(x)\,\}|\leq\frac C\lam\int_{|\nabla u|>\lam}|\nabla u(x)|dx,$

$(iii)$\,\,\,$\int_{S_h}|\nabla u-\nabla v|^2dx\leq C\int_{|\nabla u|>\lam}|\nabla u|^2dx.$

The constants $C$ above depend only on the mid-surface $S$ that are independent of $h>0.$
\end{thm}

For $a,$ $b\in\R^3,$ it is said $a<b$ if  $a_i< b_i$ for $1\leq i\leq3.$ Denote by
$$R(a,b)=(a_1,b_1)\times(a_2,b_2)\times(a_3,b_3)\qfq a,\,\,b\in\R^3\quad\mbox{with}\quad a<b,$$  the rectangle with sides parallel to the coordinate axes.

\begin{lem}(Poincar\'e's inequality)\label{l2.1}\,\,\,There exists $C>0$ such that
\be\int_{R(a,b)}|f(y)-f_{R(a,b)}|dy\leq C|b-a|\int_{R(a,b)}|\nabla f(y)|dy\label{R1}\ee
For all $a,$ $b\in\R^3$ with $a<b,$ where
$$f_{R(a,b)}=\frac1{|R(a,b)|}\int_{R(a,b)}f(y)dy={\int\hspace{-0.9em}-}_{R(a,b)}\ f(x)\, \mathrm{d}x.$$
\end{lem}
\begin{proof}1).\,\,\,Consider the Poincar$\grave{e}$ inequality for the cube $(0,1)^3:$ There exists $C>0$ such that
\be\int_{(0,1)^3}|f(y)-f_{(0,1)^3}|dy\leq C\int_{(0,1)^3}|\nabla f(y)|dy\qflq f\in\WW^{1,1}((0,1)^3).\label{R2}\ee

2).\,\,\,Consider a change of variable by
$$\var(x)=\Big((b_1-a_1)x_1+a_1,(b_2-a_2)x_2+a_2,(b_3-a_3)x_3+a_3\Big)\qfq x=(x_1,x_2,x_3)\in(0,1)^3.$$
For $f\in\WW^{1,1}(R(a,b)),$ set
$$\hat f(x)=f\circ\var(x)\qfq x\in(0,1)^3.$$ Then $\hat f\in\WW^{1,1}((0,1)^3).$
Applying  $\hat f$ to (\ref{R2}) yields (\ref{R1}) via the change of variable.
\end{proof}

For $h>0$ denote $R_h=(0,1)^2\times(-h,h).$ For $x\in\R^3$ and $r>0,$ let
$$Q(x,r)=(x_1-r,x_1+r)\times(x_2-r,x_2+r)\times(x_3-r,x_3+r)$$ be the cube with center at $x$ and side length $2r.$ Then
$$R_h\cap Q(x,r)=R(a,b),$$ where
$$a=\Big(\max\{0,x_1-r\},\max\{0,x_2-r\},\max\{-h,x_3-r\}\Big),$$
$$b=\Big(\min\{1,x_1+r\},\min\{1,x_2+r\},\min\{h,x_1+r\}\Big).$$
It is easy to check that the following two estimates are true: For $0<h<1/2,$ $x\in R_h,$ and $r>0,$
\be|R_h\cap Q(x,r)|\leq\left\{\begin{array}{l}2h\qfq r>1/2,\\
8r^2h\qfq h<r\leq1/2,\\
8r^3\qfq r\leq h,\end{array}\right.\label{R2}\ee and
\be|R_h\cap Q(x,r)|\geq\left\{\begin{array}{l}h/4\qfq r>1/2,\\
r^2h\qfq h<r\leq1/2,\\
r^3\qfq r\leq h.\end{array}\right.\ee
Furthermore, for  $x,$ $y\in R_h$ with $r=|x-y|,$ we have
\be|R_h\cap Q(x,r)\cap Q(y,r)|\geq\left\{\begin{array}{l}h/4\qfq r>1/2,\\
r^2h\qfq h<r\leq1/2,\\
r^3\qfq r\leq h,\end{array}\right.\label{R5}\ee for all $0<h<1/2.$

Using Lemma \ref{l2.1}, estimates (\ref{R2})-(\ref{R5}), and following the proof of \cite[Lemma 4.5]{Lemart}, we have the following.
\begin{lem}\label{l2.2}
For every $u\in \WW^{1,2}(R_h)$ and every
$\lam>0,$ there exists $\bar u\in\WW^{1,\infty}(\R^3)$ satisfying conditions (i) and (ii) of Theorem \ref{t2.1},
with a constant $C$ that is independent of $h\in(0,1/2).$
\end{lem}

{\bf Proof of Theorem 2.1}\,\,\,

As in \cite{FrJaMu} (iii)  directly follows from (i) and (ii). Then it will suffice to prove (i) and (ii).
First, we construct a covering on $S_h$ as follows.

Let $p_0\in \bar S$ be given. Consider a coordinate system
$\var(p)=x:$ $U\rw\R^2$ with $\var(p_0)=(1/2,1/2),$ where $U\subset M$ is  a neighborhood of $p_0.$ We define a neighborhood of $p_0$ by
$$\Om(p_0)=\var^{-1}((0,1)^2).$$
Consider the map $\Phi:$ $R_h\rw[\Om(p_0)]_h$ by
$$\Phi(x,t)=\var^{-1}(x)+tn\circ\var^{-1}(x)\qfq (x,t)\in R_h.$$ It is easy to check that $\Phi:$ $R_h\rw[\Om(p_0)]_h$ is a bilipschitz map with the Lipschitz constants of both $\Phi$ and $\Phi^{-1}$ being bounded by some $L>0$ that is independent of $h$ small.

Given $u\in \WW^{1,2}([\Om(p_0)]_h,\R^3),$ apply (i) and (ii) of Lemma \ref{l2.2} to $v=u\circ\Phi\in H^1(R_h,\R^3)$ and $L\lam$  to obtain a truncated function $\bar v\in \WW^{1,\infty}(R_h,\R^3).$ We define
$\bar u=\bar v\circ\Phi^{-1}\in\WW^{1,\infty}([\Om(p_0)]_h,\R^3)$ and observe that
\beq\|\nabla \bar u\|_{\LL^\infty([\Om(p_0)]_h,\R^3)}&&\leq L\|\nabla v\|_{\LL^\infty(R_h,\R^3)}\leq CL^2\lam,\label{2.8}\eeq and
\beq&&|\{\,z\in[\Om(p_0)]_h\,|\,u(z)\not=\bar u(z)\,\}|=|\Phi(\{\,(x,t)\in R_h\,|\,u\circ\Phi(x,t)\not=\bar u\circ\Phi(x,t)\,\})|\nonumber\\
&&\leq L^3|\{\,(x,t)\in R_h\,|\,v(x,t)\not=\bar v(x,t)\,\}|\leq\frac{CL^3}{L\lam}\int_{|\nabla v|>L\lam}|\nabla v|dz\nonumber\\
&&\leq\frac{CL^6}\lam\int_{|\nabla u|>\lam}|\nabla u|dz.\label{2.9}\eeq The constants $C$ in (\ref{2.8}) and (\ref{2.9}) are independent of $h\in(0,1/2).$

By the finite covering theorem there are $p_1,$ $\cdots,$ $p_m\in S\cup\pl S$ such that
\be\bar S\subset S',\label{2.10}\ee where $S'=\cup_{i=1}^m\Om(p_i).$ Then
$$\bar S_h\subset\cup_{i=1}^m[\Om(p_i)]_h.$$
Let $\{\,\phi_i\in\CC^2(S',[0,1])\,\}$ be a partition of unity subordinated to the cover $\{\,\Om(p_i)\,\},$ so that: $\supp\phi_i\subset\bar\Om(p_i)$ and $\sum_{i=1}^m\phi_i(p)=1$ for $p\in S'.$

Finally we follow the proof of \cite[Theorem 4.3]{Lemart} to complete our proof here. $\Box$

Let
$$\H_0^1(S_h)=\{\,u\in\WW^{1,2}(S_h)\,|\,u=0\,\,\mbox{on}\,\, (\pl S)_h\,\}.$$
\begin{pro}\label{p2.1} Let $\t{S}\subset M$ be a non-characteristic region such that $S\subset\subset\t{S}.$ Then there exists a region $S'\subset M$ with $S\subset\subset S'\subset\subset\t{S}$ such that for every $u\in \H_0^1(S_h)$ and every $\lam>0$ there exists a $\bar u\in\H_0^1(S'_h)$ satisfying $(i)-(iii)$ in Theorem $\ref{t2.1}$ with a constant $C$ is independent of $h>0$ small.
\end{pro}
\begin{proof} By the proof of Theorem \ref{t2.1} we can choose a covering $S'=\cup_{i=1}^m\Om(p_i)$ on $S$ as in (\ref{2.10}) such that $\bar{S'}\subset\t{S}.$ We extend $u\in\WW^{1,2}_0(S)$ to $u\in\WW^{1,2}_0(S')$ by
 $$u=0\qfq x\in S'\backslash S.$$ Then the proposition follows.
\end{proof}

\setcounter{equation}{0}
\section{Pointwise Rigidity Estimates of Harmonic Displacements}
\def\theequation{3.\arabic{equation}}

Let $B(x,r)\subset\R^3$ be the ball centered at $x\in\R^3$ with radius $r>0.$ Suppose that $\var\in\CC_0^\infty(B(0,2/3))$ is given such that
$$\var(x)\geq0\qfq x\in B(x,2/3);\quad\var(x)\geq\frac12\qfq x\in B(0,\frac12);\quad\int_{\R^3}\var(x)dx=1. $$
Denote
$$\var_r(x)=\frac1{r^3}\var(\frac xr)\qfq x\in\R^3.$$ If $f\in\LL^1_{\loc}(\R^3)$ with $f\geq0,$ then
\beq(\var_r*f)(x)&&=\frac1{r^3}\int_{\R^3}\var(\frac{x-z}r)f(z)dz\geq\frac1{2r^3}\int_{B(x,r/2)}f(z)dz.\label{3.1}\eeq

For $y\in\WW^{1,2}(S'_h,\R^3),$ let
\be\vare(\nabla y)=\dist(\nabla y(z)+I,\SO(3))\qfq z\in S_h.\ee
\begin{lem}There exists $C>0$ such that for any $r>0,$ any $x\in\R^3,$ and  any $y$ being harmonic on $B(x,r),$
\be \vare^2(y)(x)+r^2|\nabla^2y(x)|^2+r^4|\nabla^3y(x)|^2+r^6|\nabla^4y(x)|^2\leq C[\var_r*\vare^2(y)](x).\label{3.2} \ee
\end{lem}
\begin{proof} {\bf 1.}\,\,\,Let $x=0$ and $r=1.$ Applying Friesecke-James-M\"uller's inequality on $B(0,1)$ to the function $y+\id,$ there is $R\in\SO(3)$ such that
$$\|\nabla y+I-R\|^2_{\LL^2(B(0,1))}\leq C\|\vare(\nabla y)\|^2_{\LL^2(B(0,1))},$$ where constant $C>0$ depends only on $B(0,1).$ Since $y$ is harmonic, $\nabla y,$ $\nabla^2y,$ $\nabla^3y,$ and $\nabla^4y$ are harmonic. Then
\beq\|\nabla^4y\|^2_{\LL^2(B(0,1/16))}&&\leq C\|\nabla^3y\|^2_{\LL^2(B(0,1/8))}\leq C\|\nabla^2y\|^2_{\LL^2(B(0,1/4))}\nonumber\\
&&\leq C\|\nabla y+I-R\|^2_{\LL^2(B(0,1/2))}\leq C\|\vare(\nabla y)\|^2_{\LL^2(B(0,1))}.\nonumber\eeq It follows from the mean value theorem for harmonic functions that
\beq\vare^2(y)(0)&&\leq |\nabla y(0)+I-R|^2\leq\frac1{|B(0,1/2)|}\int_{B(0,1/2)}|\nabla y+I-R|^2dx\nonumber\\
&&\leq C\|\vare(\nabla y)\|^2_{\LL^2(B(0,1))}.\label{3.3}\eeq
Similar argument yields
\be|\nabla^2y(0)|^2+|\nabla^3y(0)|^2+|\nabla^4y(0)|^2\leq C\|\vare(\nabla y)\|^2_{\LL^2(B(0,1))}.\label{3.4}\ee

{\bf 2.}\,\,\,Fix $x\in\R^3$ and $r>0.$ Let $y$ be harmonic on $B(x,r).$ Set
$$\hat y(z)=\frac2ry(x+\frac r2z)\qfq z\in B(0,1).$$
Then (\ref{3.2}) follows by applying $\hat y$ to (\ref{3.3}), (\ref{3.4}), and (\ref{3.1}).
\end{proof}
\begin{lem}\label{l3.2}For given $\si>0$ small, let $r(x)=\si\dist(x,\pl(S_h)).$ Then there exists $C>0,$ independent of $h>0$ and $\si>0$ small, such that
\be\int_{S_h}\var_{r(x)}*f^2dx\leq C\|f\|^2_{\LL^2(S_h)}\qflq f\in\LL^2(S_h),\label{3.5}\ee where
$$\var_{r(x)}*f^2=\int_{S_h}\var_{r(x)}(x-z)f^2(z)dz.$$
\end{lem}
\begin{proof}It will suffice to prove that
$$\int_{S_h}\var_{r(x)}(x-z)dx\leq C\qflq z\in S_h.$$
Let $r_0(x)=\dist(x,\pl(S_h)).$ Since $r_0$ is Lipschitz on $S_h$ with the Lipschitz constant being $1,$ we have
$$\|\nabla r_0\|^2_{\LL^\infty(S_h)}\leq3.$$ If $x\in S_h$ is such that $|x-z|\geq r(x),$ then
$$\var_{r(x)}(x-z)=0.$$

Let $z\in S_h.$ Set
$$\Om(z)=\{\,x\in S_h\,|\,|x-z|<r(x)\,\},$$ where $r(x)=\si\dist(x,\pl(S_h)).$ We define a mapping $\psi:$ $\Om(z)\rw B(0,1)$ by
\be\psi(x)=\frac{x-z}{r(x)}\qfq x\in S_h.\label{3.6}\ee We claim that $\psi:$ $\Om(z)\rw B(0,1)$ is injective. In fact, for $x_1,$ $x_2\in\Om(z),$ we have
\beq|x_1-x_2|&&=|\psi(x_1)r(x_1)-\psi(x_2)r(x_2)|=|\psi(x_1)[r(x_1)-r(x_2)]-[\psi(x_2)-\psi(x_1)]r(x_2)|\nonumber\\
&&\leq\si|x_1-x_2|+C|\psi(x_1)-\psi(x_2)|,\nonumber\eeq that is, when $0<\si<1,$ $\psi:$ $\Om(z)\rw B(0,1)$ is injective.

From (\ref{3.6}) we have
$$\nabla\psi=\frac1r(I-\psi\otimes\nabla r).$$ Then
\be\det(r\nabla\psi)=1-\si\<\nabla r_0,\psi\>\geq1-\sqrt{3}\si>0\qfq0<\si<1/\sqrt{3}.\ee  Thus for $0<\si<1/\sqrt{3}$ the mapping
$$\psi:\quad \Om(z)\rw \psi(\Om(z))\subset B(0,1)$$ is a diffeomorphism. A change of variable $y=\psi(x)$ yields
\beq\int_{S_h}\var_{r(x)}(x-z)dx&&=\int_{\Om(z)}\frac1{r^3(x)}\var(\psi(x))dx=\int_{\psi(\Om(z))}\frac{\var(y)}{\det r\nabla\psi}dy\leq\frac1{1-\sqrt{3}\si}\nonumber\eeq for all $0<\si<1/\sqrt{3}.$ Thus (\ref{3.5}) follows.
\end{proof}

\begin{thm}\label{t3.1}
There exists a constant $C>0,$ independent of $h>0,$ such that for any $y$ being harmonic on $S_h,$
\beq&&\|\dist(\cdot,\pl(S_h))\nabla^2y\|^2_{\LL^2(S_h)}+\|\dist^2(\cdot,\pl(S_h))\nabla^3y\|^2_{\LL^2(S_h)}
\nonumber\\
&&\quad+\|\dist^3(\cdot,\pl(S_h))\nabla^4y\|^2_{\LL^2(S_h)}
\leq C\|\vare(\nabla y)\|^2_{\LL^2(S_h)}.\label{3.9}\eeq
\end{thm}
\begin{proof} Take
$$r=\frac1{2\sqrt{3}}\dist(x,\pl(S_h))\qfq x\in S_h$$ in (\ref{3.2}). We then integrate it over $S_h$ to have (\ref{3.9}) by Lemma \ref{l3.2}.
\end{proof}

\setcounter{equation}{0}
\section{Proofs of  Main Results}
\def\theequation{4.\arabic{equation}} We will need the following.
\begin{lem}\label{l4.1}Let $a>0$ be given. We have
$$\|f\|^2_{\LL^2(-a,a)}\leq 5\|f\|^2_{\LL^2(-a/2,a/2)}+4\|rf'\|^2_{\LL^2(-a,a)}\qflq f\in\WW^{1,2}(-a,a),$$
where
$$r(t)=\left\{\begin{array}{l}t+a\qfq t\in(-a,0),\\
a-t\qfq t\in[0,a).\end{array}\right.$$
\end{lem}
\begin{proof}We have
\beq\int_{-a}^xf^2(t)dt&&=(t+a)f^2(t)\Big|_{-a}^x-2\int_{-a}^x(t+a)ff'dt\nonumber\\
&&\leq (x+a)f^2(x)+\frac12\int_{-a}^xf^2(t)dt
+2\int_{-a}^x(t+a)^2f'^2(t)dt,\nonumber\eeq that is,
$$\int_{-a}^xf^2(t)dt\leq 2(x+a)f^2(x)
+4\int_{-a}^x(t+a)^2f'^2(t)dt.$$ Let $x\in(-a/2,0)$ be such that
$$f^2(x)=\frac2a\int_{-a/2}^0f^2(t)dt.$$ Then
$$\int_{-a}^{-a/2}f^2(t)dt\leq 4\int_{-a/2}^0f^2(t)dt+4\int_{-a}^0r^2(t)f'^2(t)dt.$$
We apply function $f(-\cdot)$ to the inequality above to have
$$\int_{a/2}^{a}f^2(t)dt\leq 4\int_0^{a/2}f^2(t)dt+4\int_0^ar^2(t)f'^2(t)dt.$$
Thus
\beq\int_{-a}^af^2(t)dt&&=\int_{-a}^{-a/2}f^2(t)dt+\int_{-a/2}^{a/2}f^2(t)dt+\int_{a/2}^af^2(t)dt\nonumber\\
&&\leq5\int_{-a/2}^{a/2}f^2(t)dt+4\int_{-a}^{a}r^2(t)f'^2(t)dt.\nonumber\eeq
\end{proof}

\begin{lem}\label{l4.2}Let $a>0$ be given. Then for any $s\in(-a,a)$
\be f^2(s)\leq a^{-1}\int_{-a}^af^2(t)dt+4a\int_{-a}^af'^2(t)dt\qflq f\in\WW^{1,2}(-a,a).\label{4.1*}\ee
\end{lem}
\begin{proof} Let $\si\in(-a,a)$ be such that
$$2af^2(\si)=\int_{-a}^af^2(t)dt.$$ Then
\beq f^2(s)&&=[f(\si)+\int_{\si}^sf'(t)dt]^2\leq2f^2(\si)+2(\int_{-a}^a|f'(t)|dt)^2.\nonumber\eeq
Thus (\ref{4.1*}) follows.
\end{proof}

\begin{lem}\label{l4.3} Let $\Xi=(0,a)\times(0,b),$ $0<c_1\leq c_2,$ and $f_0\in\LL^2(\Xi)$ with $c_1\leq f_0\leq c_2.$ Then there exists
$C=C(c_1,c_2)>0$ such that
\be\|f\|_{\LL^1(\Xi)}\leq C(\|f(f+f_0)\|_{\LL^1(\Xi)}+\|f(f+f_0)\|^{1/2}_{\LL^1(\Xi)}\|\nabla f\|_{\LL^2(\Xi)})\label{Xi}\ee for
all $f\in\WW^{1,2}_0(\Xi).$
\end{lem}
\begin{proof} Fix $\si=c_1/4.$ For  given $f\in\CC^1_0(\Xi),$ let
$$A_0=\{\,x\in \Xi\,|\,|f(x)|\leq\si\,\},\quad A_1=\{\,x\in \Xi\,|\,|f(x)+f_0(x)|\leq\si\,\},$$ and
$$A_2=\{\,x\in \Xi\,|\,|f(x)|>\si,\,\,|f(x)+f_0(x)|>\si\,\}.$$
Clearly, we have
$$\Xi=A_0\cup A_1\cup A_2;\quad A_i\cap A_j=\emptyset\qfq i\not=j.$$ Moreover, it follows from the definitions of
$A_0$ and $A_2$ that
\be\int_{A_0\cup A_2}|f(x)|dx\leq\si^{-1}\|f(f+f_0)\|_{\LL^1(\Xi)},\label{n4.3}\ee
\be |A_2|\leq\frac1{\si^2}\int_{A_2}|f(f+f_0)|dx\leq\si^{-2}\|f(f+f_0)\|_{\LL^1(\Xi)},\label{n4.4}\ee where $|A_2|$ denotes the $2$-dimensional Lebesgue measure of set $A_2.$

Let
$$\pi_1(A_1)=\{\,x_1\in(0,a)\,|\,(x_1,x_2)\in A_1\,\},\quad\pi_2(A_1)=\{\,x_2\in(0,b)\,|\,(x_1,x_2)\in A_1\,\} .$$ Define
$$t_+(x_1)=\inf\{\,t\in(0,b)\,|\,(x_1,t)\in A_1\}\qfq x_1\in \pi_1(A_1).$$ Then $t_+(x_1)>0$ for $x_1\in \pi_1(A_1)$ since $f|_{\pl \Xi}=0$ and
$\inf_{x\in \Xi} f_0(x)\geq c_1=4\si.$ Furthermore, let
$$t_-(x_1)=\sup\{\,t\in[0,t_+(x_1)]\,|\,(x_1,t)\in A_0\,\}\qfq x_1\in \pi_1(A_1).$$ Then
\be\{\,(x_1,x_2)\in \Xi\,|\,x_1\in \pi_1(A_1),\,\,x_2\in(t_-(x_1),t_+(x_1))\,\}\subset A_2.\label{4.4}\ee For each $x_1\in\pi_1(A_1),$ we have
\beq f(x_1,t_-(x_1))-f(x_1,t_+(x_1))&&=f(x_1,t_-(x_1))+f_0(x_1,t_+(x_1))\nonumber\\
&&\quad-[f(x_1,t_+(x_1))+f_0(x_1,t_+(x_1))]\nonumber\\
&&\geq c_1-2\si=2\si>0.\nonumber\eeq From (\ref{4.4}) we obtain
\beq2\si|\pi_1(A_1)|&&\leq\int_{\pi_1(A_1)}[f(x_1,t_-(x_1))-f(x_1,t_+(x_1))]dx_1=\int_{\pi_1(A_1)}
\int_{t_+(x_1)}^{t_-(x_1)}f_{x_2}(x_1,x_2)dx_2dx_1\nonumber\\
&&\leq \int_{A_2}|\nabla f|dx\leq|A_2|^{1/2}\|\nabla f\|_{\LL^2(\Xi)},\label{4.5}\eeq where $|\pi_1(A_1)|$ is the $1$-dimensional Lebesgue measure of set $\pi_1(A_1).$
A similar argument for $\pi_2(A_1)$ as in (\ref{4.5})  yields
\be2\si|\pi_2(A_1)|\leq\int_{A_2}|\nabla f|dx\leq|A_2|^{1/2}\|\nabla f\|_{\LL^2(\Xi)}.\label{4.6}\ee
By (\ref{4.5}) and (\ref{4.6}) we then obtain
$$|A_1|\leq|\pi_1(A_1)||\pi_2(A_1)|\leq 4c_1^{-2}|A_2|\|\nabla f\|^2_{\LL^2(\Xi)}.$$ It follows from (\ref{n4.4}) that
\beq\int_{A_1}f^2dx&&\leq2\int_{A_1}(|f+f_0|^2+f_0^2)dx\leq2(\si^2+c_2^2)|A_1|\leq C|A_2|\|\nabla f\|^2_{\LL^2(\Xi)}\nonumber\\
&&\leq C\|f(f+f_0)\|_{\LL^1(\Xi)}\|\nabla f\|^2_{\LL^2(\Xi)}.\label{4.7}\eeq

Finally, from (\ref{n4.3}) and (\ref{4.7}) we obtain
\beq\|f\|_{\LL^1(\Xi)}&&=\int_{A_0\cup A_2}|f|dx+\int_{A_1}|f|dx\leq\si^{-1}\|f(f+f_0)\|_{\LL^1(\Xi)}+|A_1|^{1/2}\|f\|_{\LL^2(A_1)}\nonumber\\
&&\leq C\|f(f+f_0)\|_{\LL^1(\Xi)}+C\|f(f+f_0)\|^{1/2}_{\LL^1(\Xi)}\|\nabla f\|_{\LL^2(\Xi)}.\nonumber\eeq
\end{proof}

For $A\in\R^{3\times3},$ let
$$B=A+I,\quad \Phi(A)=\sym A+\frac12A^TA.$$  It is easy to check that
\begin{lem}
\be|\Phi(A)|\leq\frac{\sqrt{3}+|B|}2\dist(B,\SO(3)).\label{4.12**}\ee  Moreover, if $\det B>0,$ then
\be  \dist(B,\SO(3))\leq\sqrt{3}|B^TB-I|=2\sqrt{3}|\Phi(A)|.\ee
\end{lem}

We extend the domain $S_h$ of $u\in\H_0^1(S_h,\R^3)$ to be $M_h,$ still denoted as $u,$ by
$$u=0\qfq z=x+tn(x)\in M_h/S_h.$$

Let $S'\subset M$ be the region given by Proposition \ref{p2.1} with $S\subset\subset S'\subset\subset\t S.$ Then $u\in\H_0^1(S'_h,\R^3).$ Let $\bar{u}\in\H_0^1(S'_h,\R^3)$ be given in Proposition \ref{p2.1} satisfying (i)-(iii) in Theorem \ref{t2.1} with
respect to $\lam=3\sqrt{4}.$ Then  there exists $C>0,$ independent of $h>0,$ such that
$$\|\nabla\bar{u}\|_{\LL^\infty(S'_h)}\leq C,$$
$$\|\nabla u-\nabla\bar{u}\|^2_{\LL^2(S'_h)}\leq C\int_{|\nabla u|>3\sqrt{3}}|\nabla u|^2dz.$$
We denote
 $$y_1=\bar{u}+\id,\quad \vare(\nabla u)=\dist(\nabla u+I,\SO(3)).$$ It follows from the inequalities above that
\be\|\nabla y_1\|_{\LL^\infty(S'_h)}\leq C,\quad \|\nabla\bar{u}-\nabla u\|^2_{\LL^2(S'_h)}\leq C\|\vare(\nabla u)\|^2_{\LL^2(S_h)},\label{4.15*}\ee where constant $C$ is independent of $h>0.$

We solve problem
\be\left\{\begin{array}{l}\Delta y_2=\Delta y_1\qfq z\in \t{S}_h,\\
y_2|_{\pl(\t{S}_h)}=0.\end{array}\right.\label{4.16*}\ee Define
\be y=y_1-y_2-\id.\label{4.17*}\ee Then $y$ is harmonic on $\t{S}_h.$ Since $y_1|_{(\t{S}\backslash S')_h}=\id,$ we have
$$y|_{(\pl\t{S})_h}=0,$$ that is, $y\in\H_0^1(\t{S}_h,\R^3).$

\begin{lem}\label{l4.6} Let $y$ be given in $(\ref{4.17*})$ from $u\in\H_0^1(S_h,\R^3).$ Then there exists $C>0$ such that
\beq\Big|\|\nabla u\|_{\LL^2(S_h)}-\|\nabla y\|_{\LL^2(\t{S}_h)}\Big|&+&\|\nabla y\|_{\LL^2((\t{S}\backslash S')_h)}+\|\vare(\nabla y)\|_{\LL^2(\t{S}_h)}\nonumber\\
&\leq& C\|\vare(\nabla u)\|_{\LL^2(S_h)},\label{4.18**}\eeq where $\vare(\nabla y)=\dist(\nabla y+I,\SO(3))$ and constant $C$ is
independent of $u\in\H_0^1(S_h,\R^3)$ and $h>0.$
\end{lem}

\begin{proof} From (\ref{4.16*}) and (\ref{4.15*}) we have
\beq\|\nabla y_2\|^2_{\LL^2(\t{S}_h)}&&=-\int_{\t{S}_h}\<y_2,\Delta y_2\>dz=-\int_{\t{S}_h}\<y_2,\div(\nabla y_1
-\cof\nabla y_1)\>dz\nonumber\\
&&\leq\frac12\|\nabla y_2\|^2_{\LL^2(\t{S}_h)}+C\|\dist(\nabla y_1,\SO(3))\|^2_{\LL^2(\t{S}_h)},\nonumber\eeq
that gives
\beq\|\nabla y_2\|_{\LL^2(\t{S}_h)}&&\leq C\|\dist(\nabla y_1,\SO(3))\|_{\LL^2(\t{S}_h)}\nonumber\\
&&\leq C\|\vare(\nabla u)\|_{\LL^2(S_h)}+\|\nabla\bar{u}-\nabla u\|_{\LL^2(S'_h)}\leq C\|\vare(\nabla u)\|_{\LL^2(S_h)}\label{4.18*}\eeq
It follows from $y_1|_{(\t{S}\backslash\,S')_h}=\id$ and (\ref{4.18*}) that
\be\|\nabla y\|_{\LL^2((\t{S}\backslash\,S')_h)}=\|\nabla y_2\|_{\LL^2((\t{S}\backslash\,S')_h)}\leq C\|\vare(\nabla u)\|_{\LL^2(S_h)}.\label{4.19*}\ee

Since
$$\nabla y=\nabla u+\nabla\bar{u}-\nabla u-\nabla y_2,$$ from (\ref{4.15*}) and (\ref{4.18*}) we have
\be \|\vare(\nabla y)\|_{\LL^2(\t{S}_h)}\leq C\|\vare(\nabla u)\|_{\LL^2(S_h)}+\|\nabla\bar{u}-\nabla u\|_{\LL^2(S'_h)}
+\|\nabla y_2\|_{\LL^2(\t{S}_h)}\leq C\|\vare(\nabla u)\|_{\LL^2(S_h)}.\label{4.20*}\ee
Similarly, we have
\beq&&\Big|\|\nabla u\|_{\LL^2(S_h)}-\|\nabla y\|_{\LL^2(\t{S}_h)}\Big|\leq\|\nabla y-\nabla u\|_{\LL^2(\t{S}_h)}\nonumber\\
&&=\|\nabla\bar u-\nabla u-\nabla y_2\|_{\LL^2(\t{S}_h)}\leq C\|\vare(\nabla u)\|_{\LL^2(S_h)}.\label{4.21*}\eeq

Finally, (\ref{4.18**}) follows from (\ref{4.21*}), (\ref{4.20*}), and (\ref{4.19*}).
\end{proof}

{\bf A Cutoff Function.}\,\,\,We fix a region $S''\subset M$ such that
\be S'\subset\subset S''\subset\subset \t{S}.\label{S''}\ee Suppose that $\psi\in\CC^2(S'')$ is a cutoff function such that \be\psi|_{S'}=1,\quad \supp\psi\subset S''.\label{psi}\ee

Let $f\in\LL^1(S_h).$ Since
$$\int_{S_h}|f(z)|dz=\int_{-h}^h\int_S|f(x+tn(x))|(1+t\tr\nabla n+t^2\kappa)dgdt,$$ where $\kappa$ is the Gaussian curvature, we have
$$(1-Ch)\int_{-h}^h\int_S|f|dgdt\leq\int_{S_h}|f(z)|dz\leq(1+Ch)\int_{-h}^h\int_S|f|dgdt.$$ Thus, if $h>0$ is small,  we use
$$\int_{-h}^h\int_S|f|dgdt$$ to replace
$$\int_{S_h}|f(z)|dz.$$

\begin{lem}\label{l4.7} Let region $S''$ and function $\psi$ be given in $(\ref{S''})$ and $(\ref{psi}),$ respectively. Suppose that $y$ is the harmonic function given in $(\ref{4.17*})$ from $u\in\H_0^1(S_h,\R^3).$ Then there exist $C>0$ and $h_0>0$ such that
\be\|\vare(\nabla(\psi y))\|_{\LL^2(\t{S}_h)}\leq C\|\vare(\nabla u)\|_{\LL^2(S_h)},\,\,\|\nabla^2(\psi y)\|^2_{\LL^2(\t{S}_{h/2})}\leq Ch^{-2}\|\vare(\nabla u)\|^2_{\LL^2(S_h)},\label{4.27}\ee where $C>0$ is independent of
$u\in\H_0^1(S_h,\R^3)$ and $0<h<h_0.$
\end{lem}
\begin{proof} By (\ref{4.18**}) we have
\beq\|\vare(\nabla(\psi y))\|_{\LL^2(\t{S}_h)}&&=\|\dist(\nabla y+(\psi-1)\nabla y+y\otimes\nabla\psi+I,\SO(3))\|_{\LL^2(\t{S}_h)}\nonumber\\
&&\leq\|\vare(\nabla y)\|_{\LL^2(\t{S}_h)}+C\|\nabla y\|_{\LL^2((\t{S}\backslash\,S')_h)}+C\| y\|_{\LL^2((\t{S}\backslash\,S')_h)}\nonumber\\
&&\leq C\|\vare(\nabla u)\|_{\LL^2(S_h)}+C\| y\|_{\LL^2((\t{S}\backslash S')_h)}.\label{4.24}\eeq
Next, we estimate the term $\|y\|_{\LL^2((\t{S}\backslash\,S')_h)}.$ By Poincar\'e's inequality on $\t{S}\backslash\,S',$ we have
$$\int_{\t{S}\backslash S'}|y|^2dg\leq C\int_{\t{S}\backslash\,S'}|Dy|^2dg\leq C\int_{\t{S}\backslash\,S'}|\nabla y|^2dg,$$ where $D$ denotes the covariant differential of the induced metric $g$ of surface $M.$
We integrate the inequality above over $(-h,h)$ with respect to $t$ to have
\be\|y\|^2_{\LL^2((\t{S}\backslash\,S')_{h})}\leq C\|\nabla y\|^2_{\LL^2((\t{S}\backslash\,S')_{h})}.\label{4.25}\ee Then the first estimate in (\ref{4.27})  follows from (\ref{4.24}) and (\ref{4.18**}).

Let $h_0=\dist(S'',\pl\t{S}).$  Since $y$ is harmonic on $\t{S},$ using (\ref{4.18**}), (\ref{4.25}), and (\ref{3.9}), we have, for $0<h<h_0,$
\beq\|\nabla^2(\psi y)\|^2_{\LL^2(\t{S}_{h/2})}&&\leq C\|\nabla^2y\|^2_{\LL^2(S''_{h/2})}
+C\|\nabla y\|^2_{\LL^2((\t{S}\backslash\,S)_{h/2})}+C\|y\|^2_{\LL^2((\t{S}\backslash\,S)_{h/2})}\nonumber\\
&&\leq Ch^{-2}\|\dist(\cdot,\pl(\t{S}_h))\nabla^2y\|^2_{\LL^2(\t{S}_{h})}
+C\|\vare(\nabla u)\|^2_{\LL^2(S_h)}\nonumber\\
&&\leq Ch^{-2}\|\vare(\nabla y)\|^2_{\LL^2(\t{S}_h)}+C\|\vare(\nabla u)\|^2_{\LL^2(S_h)}\nonumber\\
&&\leq Ch^{-2}\|\vare(\nabla u)\|^2_{\LL^2(S_h)}.\nonumber\eeq
\end{proof}

\begin{thm}\label{t4.1} Let $S$ be a bounded $\CC^2$ Lipschitz region. Then there exist constants $C>0$ and $h_0>0$ such that
\be\|\nabla u\|^2_{\LL^2(S_h,\R^3)}\leq\frac C{h^2}\|\vare(\nabla u)\|^2_{\LL^2(S_h)}\label{4.n28}\ee for all $u\in\H_0^1(S_h,\R^3)$ and all $0<h<h_0.$
\end{thm}
\begin{proof}Let region $S''$ and function $\psi$ be given in $(\ref{S''})$ and $(\ref{psi}),$ respectively. Suppose that $y$ is the harmonic function given in $(\ref{4.17*})$ from $u\in\H_0^1(S_h,\R^3).$ Using the second inequality in (\ref{4.27}), we have
$$\|\nabla y\|_{\LL^2(S_{h/2})}\leq Ch^{-1}\|\vare(\nabla u)\|_{\LL^2(S_h)}.$$ By Lemma \ref{l4.1}  and Theorem \ref{t3.1}, we obtain
\beq\|\nabla y\|^2_{\LL^2(S_h)}&&\leq 5\|\nabla y\|^2_{\LL^2(S_{h/2})}+C\|\dist(\cdot,\pl\t{S}_{h/2})\nabla^2 y\|^2_{\LL^2(S_{h/2})}\nonumber\\
&&\leq Ch^{-2}\|\vare(\nabla u)\|^2_{\LL^2(S_h)}+C\|r\nabla^2 y\|^2_{\LL^2(\t{S}_{h/2})}\leq Ch^{-2}\|\vare(\nabla u)\|^2_{\LL^2(S_h)}.\nonumber\eeq Then it follows from Lemma \ref{l4.6} that
$$\|\nabla y\|^2_{\LL^2(\t{S}_h)}=\|\nabla y\|^2_{\LL^2(S_h)}+\|\nabla y\|^2_{\LL^2((\t{S}\backslash S)_h)}
\leq Ch^{-2}\|\vare(\nabla u)\|^2_{\LL^2(S_h)}.$$
Thus (\ref{4.n28}) follows from Lemma \ref{l4.6} again.
\end{proof}

 For $y\in\WW^{1,2}(\t{S}_h,R^3),$  We define
\be Dy(z)=\sum_{ij=1}^2\nabla_{e_i}y(z)\otimes e_j\qfq z=x+tn(x)\in\t{S}_h,\label{D}\ee where $\{e_1,e_2\}$ is an orthonormal basis of $T_x\t{S}.$ Then
$$|\nabla y|^2=|Dy|^2+|y_t|^2\qfq z=x+tn(x)\in\t{S}_h.$$

{\bf Assumption (H).}\,\,\,Let $\tau>0$ be such that the following estimates hold for the Korn inequality:
\be\|\nabla u\|^2_{\LL^2(S_h)}\leq \frac C{h^\tau}\|\sym\nabla u\|^2_{\LL^2(S_h)}\label{4.9*}\ee for all $u\in \H_0^1(S_h,\R^3).$

\begin{thm}\label{t4.2}Suppose assumption $(H)$ holds true. Let $\delta>0$ and $\a>0$ be given. There exist $C>0$ and $h_0>0$ such that, if $u\in\H_0^1(S_h,\R^3)$ satisfies
\be\|\vare(\nabla u)\|_{\LL^2(S_h)}^2\leq\delta h^{\tau+3+\a}\qfq 0<h<h_0,\label{4.29*}\ee then
\be\|\nabla u\|^2_{\LL^2(S_h)}\leq Ch^{-\tau}\|\vare(\nabla u)\|^2_{\LL^2(S_h)}\label{4.33}\ee for all $u\in\H_0^1(S_h,\R^3)$ and $0<h<h_0.$
\end{thm}
\begin{proof} Let $y$ be the harmonic function given in $(\ref{4.17*})$ from $u.$ Let $\theta\in(0,1)$ be given. By the interpolation inequality, (\ref{4.1*}),  (\ref{3.9}), and (\ref{4.18**}), there exists $C>0$ such that
\beq&&\|Dy\|^2_{\WW^{1+\theta,2}({S},\R^3)}\leq C\|Dy\|^{2(1-\theta)}_{\WW^{1,2}({S},\R^3)}\|D y\|^{2\theta}_{\WW^{2,2}({S},\R^3)}\nonumber\\
&&\leq C(h^{-1}\|\nabla y\|^2_{\WW^{1,2}({S}_{h/2},\R^3)}+h\|\nabla^2y\|^2_{\WW^{1,2}({S}_{h/2},\R^3)})^{1-\theta}\nonumber\\
&&\quad(h^{-1}\|\nabla y\|^2_{\WW^{2,2}({S}_{h/2},\R^3)}+h\|\nabla^2y\|^2_{\WW^{2,2}(S_{h/2},\R^3)})^{\theta}\nonumber\\
&&\leq  C[h^{-1}\|\nabla y\|^2_{\LL^2({S}_{h/2},\R^3)}+(h^{-1}+h)\|\nabla^2 y\|^2_{\LL^2({S}_{h/2},\R^3)}+h\|\nabla^3 y\|^2_{\LL^2(S_{h/2},\R^3)}]^{1-\theta}\nonumber\\
&&\quad\{h^{-1}\|\nabla y\|^2_{\LL^2({S}_{h/2},\R^3)}+(h^{-1}+h)(\|\nabla^2 y\|^2_{\LL^2({S}_{h/2},\R^3)}+\|\nabla^3 y\|^2_{\LL^2({S}_{h/2},\R^3)})\nonumber\\
&&\quad+h\|\nabla^4y\|^2_{\LL^2({S}_{h/2},\R^3)}\}^{\theta}\nonumber\\
&&\leq  C[h^{-1}\|\nabla y\|^2_{\LL^2({S}_{h},\R^3)}+(h^{-1}+h)h^{-2}\|r\nabla^2 y\|^2_{\LL^2(\t{S}_{h},\R^3)}+h^{-3}\|r^2\nabla^3 y\|^2_{\LL^2(\t{S}_{h},\R^3)}]^{1-\theta}\nonumber\\
&&\quad\{h^{-1}\|\nabla y\|^2_{\LL^2(\t{S}_{h},\R^3)}+(h^{-1}+h)(h^{-2}\|r\nabla^2 y\|^2_{\LL^2(\t{S}_{h},\R^3)}+h^{-4}\|r^2\nabla^3 y\|^2_{\LL^2(\t{S}_{h},\R^3)})\nonumber\\
&&\quad+h^{-5}\|r^3\nabla^4y\|^2_{\LL^2(\t{S}_{h},\R^3)}\}^{\theta}\nonumber\\
&&\leq Ch^{-3-2\theta}\|\vare(\nabla u)\|^2_{\LL^2(S_h)}\nonumber\eeq for all $0<h<\dist(S,\pl\t{S}),$ $t\in(-h/2,h/2),$ and $1\leq i\leq 3,$ where
$r=\dist(\cdot,\pl(\t{S}_h)).$ A similar argument yields
$$\|y_t\|^2_{\WW^{1+\theta,2}({S},\R^3)}\leq Ch^{-3-2\theta}\|\vare(\nabla u)\|^2_{\LL^2(S_h)}.$$ Thus we obtain
$$\|\nabla y\|^2_{\WW^{1+\theta,2}({S},\R^3)}\leq Ch^{-3-2\theta}\|\vare(\nabla u)\|^2_{\LL^2(S_h)}. $$
The Sobolev embedding theorem yields
$$\|\nabla y\|_{\LL^\infty(S)}\leq C_\theta h^{-3/2-\theta}\|\vare(\nabla u)\|_{\LL^2(S_h,\R^3)}\qfq t\in(-h/2,h/2).$$ Thus
\be\|\nabla y\|_{\LL^\infty(S_{h/2})}\leq C_\theta h^{-3/2-\theta}\|\vare(\nabla u)\|_{\LL^2(S_h)}.\label{4.29}\ee

Since $y\in\H_0^1(\t{S}_{h/2},\R^3),$ from (\ref{4.9*}), (\ref{4.12**}), (\ref{4.18**}), and (\ref{4.29}) we obtain
\beq\|\nabla y\|_{\LL^2(\t{S}_{h/2})}^2
&&\leq Ch^{-\tau}\|\sym\nabla y\|^2_{\LL^2(\t{S}_{h/2})}\nonumber\\
&&\leq Ch^{-\tau}\|\sym\nabla y\|^2_{\LL^2(S_{h/2})}+Ch^{-\tau}\|\nabla y\|^2_{\LL^2((\t{S}/S)_{h/2})}\nonumber\\
&&\leq C h^{-\tau}\|\Phi(\nabla y)\|^2_{\LL^2(S_{h/2})}+C_\theta h^{-\tau-3-2\theta}\|\vare(\nabla u)\|^2_{\LL^2(S_h)}\|\nabla y\|^2_{\LL^2(S_{h/2})}\nonumber\\
&&\quad+Ch^{-\tau}\|\vare(\nabla u)\|^2_{\LL^2(S_{h})}\nonumber\\
&&\leq CC_\theta h^{-\tau-3-2\theta}\|\vare(\nabla u)\|^2_{\LL^2(S_h)}(\|\nabla y\|^2_{\LL^2(\t{S}_{h/2})}+\|\vare(\nabla u)\|^2_{\LL^2(S_h)})\nonumber\\
&&\quad+Ch^{-\tau}\|\vare(\nabla u)\|^2_{\LL^2(S_{h})} \label{4.32}\eeq

Let
$$\theta=\frac\a3,\quad h_0=(\frac1{2CC_\theta\delta})^{3/\a}.$$
By the assumption (\ref{4.29*}), when $0<h<h_0$ we have
$$CC_\theta h^{-\tau-3-2\theta}\|\vare(\nabla u)\|_{\LL^2(S_h)}^2\leq CC_\theta\delta h^{\a/3}<\frac12.$$ It follows from (\ref{4.32}) that
\be\|\nabla y\|_{\LL^2(\t{S}_{h/2})}^2\leq Ch^{-\tau}\|\vare(\nabla u)\|^2_{\LL^2(S_{h})}.\label{4.34}\ee

Let  region $S''$ be given in (\ref{S''}). Using Lemma \ref{l4.1}, (\ref{4.34}), and (\ref{3.9}), we have
\beq\|\nabla y\|_{\LL^2(S''_h)}^2&&\leq5\|\nabla y\|^2_{\LL^2(S''_{h/2})}+4\int_{-h}^hr^2_0(t)\|\nabla^2y\|^2_{\LL^2(S'')}dt\nonumber\\
&&\leq5\|\nabla y\|^2_{\LL^2(S''_{h/2})}+4\|r\nabla^2y\|^2_{\LL^2(\t{S}_h)}\leq Ch^{-\tau}\|\vare(\nabla u)\|^2_{\LL^2(S_{h})}\label{4.36n}\eeq for $0<h<\dist(S'',\pl\t{S}),$ where $r=\dist(\cdot,\pl(\t{S}_h))$ and
$$r_0(t)=\left\{\begin{array}{l}t+h\qfq t\in(-h,0),\\
h-t\qfq t\in[0,h).\end{array}\right.$$

Moreover, we replace $\t{S}$ with $S''$ in Lemma \ref{l4.6} to have
\be\|\nabla u\|^2_{\LL^2(S_h)}\leq \|\nabla y\|^2_{\LL^2(S''_h)}+C\|\vare(\nabla u)\|^2_{\LL^2(S_h)}.\label{4.37n}\ee
Combining (\ref{4.36n}) and (\ref{4.37n}), we obtain (\ref{4.33}).
\end{proof}

We also need the following lemma.
\begin{lem}$(\cite{Yao2020})$ Let $y=F+fn\in\WW^{1,2}(S_h,\R^3)$  be given with $f=\<F,n\>$ and let $A=\nabla y(I+t\nabla n).$ Then
\be|A|^2=|DF+f\n|^2+|Df-\nabla nF|^2+|F_t|^2+f_t^2,\label{4.36}\ee
\be|\sym A|^2=|\Upsilon(y)|^2+|\X(y)|^2+f_t^2\label{4.37}\ee
for $z=x+tn(x)\in S_h,$ where
$$\Upsilon(y)=\sym DF+f\nabla n,\quad \X(y)=Df-\nabla nF+F_t.$$
\end{lem}

Let region $S''$ and function $\psi$ be given in  (\ref{S''}) and (\ref{psi}), respectively. Suppose $y$ is the harmonic function given in $(\ref{4.17*})$ from $u.$
Set
\be\psi y=W+wn,\quad A=\nabla(\psi y)(I+t\nabla n),\label{n4.36}\ee where $w=\<\psi y,n\>.$ Then
\be(1-Ch)|A|\leq|\nabla(\psi y)|\leq (1+Ch)|A|\qfq z=x+tn(x)\in\t{S}_h.\label{4.38*}\ee
It is easy to check that the following identities hold true:
\be AX=D_X W+w\nabla nX+[X(w)-\<\nabla nW,X\>]n\qfq X\in T_x\t{S},\label{4.38}\ee
\be A^TX=(D^TW+w\nabla n)X+\<X,W_t\>n\qfq X\in T_x\t{S},\label{4.39}\ee and
\be An=W_t+w_tn,\quad A^Tn=Dw-\nabla nW+w_tn\qfq z=x+tn\in \t{S}_h.\label{4.40}\ee

Let
$$\E(A)=\sqrt{(A^T+I)(A+I)}-I,\quad \E'(A)=\sqrt{(A+I)(A^T+I)}-I.$$ Then
\be\E^2+2\E=2\Phi(A),\quad\E'^2+2\E'=2\Phi(A^T),\quad|\E(A)|\leq\vare(A),\quad|\E'(A)|\leq\vare(A),\quad\label{n4.40}\ee where
$$\vare(A)=\dist(A+I,\SO(3)),\quad \Phi(A)=\sym A+\frac12A^TA.$$
Let $\{E_1,E_2\}$ be an orthonormal basis of $T_x\t{S}.$ Then $\{E_1,E_2,n\}$ forms an orthonormal basis of $\R_z^3.$ By (\ref{4.38}), we have
\beq&&\sum_{i=1}^2\<(\E^2+2\E)E_i,E_i\>=2\sum_{i=1}^2\<AE_i,E_i\>+\sum_{i=1}^2|AE_i|^2\nonumber\\
&&=2\div_gW+2w\tr_g\nabla n+|DW+w\nabla n|^2+|Dw-\nabla nW|^2,\label{4.41}\eeq
and, by (\ref{4.39}),
\be\sum_{i=1}^2\<(\E'^2+2\E')E_i,E_i\>=2\div_gW+2w\tr_g\nabla n+|DW+w\nabla n|^2+|W_t|^2,\label{4.42}\ee for $z=x+tn(x)\in \t{S}_h.$

\begin{lem} We have
\be \|\vare(A)\|^2_{\LL^2(\t{S}_h)}\leq C\|\vare(\nabla u)\|^2_{\LL^2(S_h)}.\label{n3.44}\ee
\end{lem}
\begin{proof} Since
\beq|A+I-Q|&&\leq|\nabla(\psi y)+I-Q|+Ch|\nabla (\psi y)|\qflq Q\in\SO(3),\nonumber\eeq
it follows from Lemma \ref{l4.7} and Theorem \ref{t4.1} that
\beq\|\vare(A)\|^2_{\LL^2(\t{S}_h)}&&\leq2\|\vare(\nabla(\psi y))\|^2_{\LL^2(\t{S}_h)}+ Ch^2\|\nabla(\psi y)\|^2_{\LL^2(\t{S}_h)}\nonumber\\
&&\leq C\|\vare(\nabla(\psi y))\|^2_{\LL^2(\t{S}_h)}\leq C\|\vare(\nabla u)\|^2_{\LL^2(S_h)}.\nonumber\eeq
\end{proof}

\begin{lem} Let $S$ be the ruled surface, given by $(\ref{om}).$  Then there exists $C>0,$ independent of $h>0$ small, such that
\be\|DW+w\nabla n\|^2_{\LL^2(S''_{h/2})}+\|Dw-\nabla nW\|^2_{\LL^2(S''_{h/2})}\leq CP(u,h),\label{n3.45}\ee
where
\beq P(u,h)&&=\|\vare(\nabla u)\|^2_{\LL^2(S_h)}+h^{1/2}\|\vare(\nabla u)\|_{\LL^2(S_h)}\nonumber\\
&&\quad+(\|\vare(\nabla u)\|^2_{\LL^2(S_h)}+h^{1/2}\|\vare(\nabla u)\|_{\LL^2(S_h)})^{1/2}h^{-1}\|\vare(\nabla u)\|_{\LL^2(S_h)}.\label{n3.46}\eeq
\end{lem}
\begin{proof} {\bf 1.}\,\,\,Since
$$\pl_1=\frac{\pl\om(s)}{\pl s_1}=c'(s_1)+s_2\delta'(s_1),\quad \pl_2=\frac{\pl\om(s)}{\pl s_2}=\delta(s_1),$$ we have $\nabla_{\pl_2}\pl_2=0.$ Then
\be D_{\pl_2}\pl_2=0,\quad \<\nabla_{\pl_2}n,\pl_2\>=0\qoq \bar{\t{S}}.\label{n4.43}\ee

From (\ref{4.38}) and (\ref{n4.43}) we have
\beq 2\<\Phi(A)\pl_2,\pl_2\>&&=2\<A\pl_2,\pl_2\>+|A\pl_2|^2=2\<(DW+w\nabla n)\pl_2,\pl_2\>+|A\pl_2|^2\nonumber\\
&&=2\frac{\pl}{\pl s_2}\<W,\pl_2\>+|A\pl_2|^2.\nonumber\eeq
Denote $z(s)=\om(s)+tn(\om(s)).$ It follows from $W\in\H_0^1(\t{S}_h,\R^3)$ and (\ref{n4.40}) that
\beq\int_{b_1}^{b_2}|\psi y(z(s))|^2ds_2&&\leq C\int_{b_1}^{b_2}\Big|\frac{\pl}{\pl s_2}(\psi y)(z(s))\Big|^2ds_2
=C\int_{b_1}^{b_2}\Big|A\pl_2\Big|^2ds_2\nonumber\\
&&\leq C\int_{b_1}^{b-2}|\Phi(A)|ds_2\leq C\int_0^b(|\E(A)|^2+|\E(A)|)ds_2.\nonumber\eeq
We then integrate the above inequality with respect to $s_1$ over $(a_1, a_2)$ to have
$$\|\psi y\|^2_{\LL^2(S'')}\leq C\int_{S''}(|\E(A)|^2+|\E(A)|)dg. $$
Finally, integrating the above inequality with respect to $t$ over $(-h/2,h/2)$ yields, by (\ref{n3.44}),
\beq\|\psi y\|^2_{\LL^2(S''_{h/2})}&&\leq C(\|\vare(A)\|^2_{\LL^2(\t{S}_h)}+h^{1/2}\|\vare(A)\|_{\LL^2(\t{S}_h)})\nonumber\\
&&\leq C(\|\vare(\nabla u)\|^2_{\LL^2(S_h)}+h^{1/2}\|\vare(\nabla u)\|_{\LL^2(S_h)}).\label{n3.48} \eeq

{\bf 2.}\,\,\,Using the interpolation theorem and the Poincar\'e inequality over region $S'',$ we have
\beq&&\|D(\psi y)\|^2_{\LL^2(S'')}\leq C\|\psi y\|_{\LL^2(S'')}\|\psi y\|_{\WW^{2,2}(S'')}\nonumber\\
&&\leq C\|\psi y\|_{\LL^2(S'')}(\|\psi y\|_{\LL^2(S'')}+\|D(\psi y)\|_{\LL^2(S'')}+\|D^2(\psi y)\|_{\LL^2(S'')})\nonumber\\
&&\leq C\|\psi y\|_{\LL^2(S'')}(\|D(\psi y)\|_{\LL^2(S'')}+\|D^2(\psi y)\|_{\LL^2(S'')}),\nonumber\eeq
which yields
$$\|D(\psi y)\|^2_{\LL^2(S'')}\leq C\|\psi y\|_{\LL^2(S'')}^2+C\|\psi y\|_{\LL^2(S'')}\|D^2(\psi y)\|_{\LL^2(S'')}.$$ Integrating the above inequality with respect to $t$ over $(-h/2,h/2)$ gives, by (\ref{n3.48}) and (\ref{4.27}),
\beq\|D(\psi y)\|^2_{\LL^2(S''_{h/2})}&&\leq C\|\psi y\|_{\LL^2(S''_{h/2})}^2+C\|\psi y\|_{\LL^2(S''_{h/2})}\|\nabla^2(\psi y)\|_{\LL^2(S''_{h/2})}\nonumber\\
&&\leq CP(u,h).\label{n3.49}\eeq

{\bf 3.}\,\,\,Fix $x\in S''.$ Let $\{e_1,e_2\}$ be an orthonormal basis of $T_xS''$ such that
$$\nabla_{e_i}n=\lam_ie_i\qaq x\qfq i=1,\,\,2,$$ where $\lam_1,$ $\lam_2$ are the principal curvatures.
By (\ref{D}), (\ref{n4.36}) and (\ref{4.38}), we have
\beq|D(\psi y)|^2&&=\sum_{ij=1}^2\<e_i,\nabla(\psi y)e_j\>^2+\sum_{i=1}^2\<n,\nabla_{e_i}(\psi y)\>^2\nonumber\\
&&=\sum_{ij=1}^2\frac1{(1+t\lam_j)^2}\<e_i,Ae_j\>^2+\sum_{i=1}^2\frac1{(1+t\lam_i)^2}\<n,Ae_i\>^2\nonumber\\
&&\geq \si\sum_{i=1}^2|Ae_i|^2=\si(|DW+w\nabla n|^2+|Dw-\nabla nW|^2),\label{n3.50}\eeq for  all $z=x+tn(x)\in S''_{h/2}$ and $0<h<h_0$ where  $\si>0$ and $h_0>0$ are small constants. Thus (\ref{n3.45}) follows from (\ref{n3.49}) and (\ref{n3.50}).
\end{proof}

\begin{lem}\label{l4.10} Suppose one of the following assumptions $(a)$ or $(b)$ holds:

$(a)$\,\,\,$S$ is a minimal surface, and

$(b)$\,\,\,$S$ is a ruled surface, given in $(\ref{om}).$

Then there is $C>0,$ independent of $h>0$ small, such that
\be\|\nabla(\psi y)\|^2_{\LL^2(S''_{h/2})}\leq CP(u,h),\label{n4.51}\ee where $P(u,h)$ is given in $(\ref{n3.46}).$
\end{lem}

\begin{proof}
(a)\,\,\,Let $S$ be a minimal surface with $\tr_g\nabla n=0.$ Since $W\in\WW_0^{1,2}(S'',\R^3),$ using (\ref{4.41}) and (\ref{4.42}), we obtain
\beq &&\|DW+w\nabla n\|^2_{\LL^2(S''_h)}+\|Dw-\nabla nw\|^2_{\LL^2(S''_h)}+\|W_t\|_{\LL^2(S''_h)}^2\nonumber\\
&&\leq C\int_{S''_h}(|\E|^2+|\E'|^2+|\E|+|\E'|)dz\nonumber\\
&&\leq C\|\vare(A)\|^2_{\LL^2(S''_h)}+Ch^{1/2}\|\vare(A)\|_{\LL^2(S''_h)}\qfq t\in(-h,h).\label{4.44}\eeq

Moreover, by (\ref{4.40}), we have
\beq\<(\E^2+2\E)n,n\>&&=\<(A+A^T+A^TA)n,n\>=2\<n,An\>+|An|^2\nonumber\\
&&=2w_t+w_t^2+|W_t|^2.\nonumber\eeq Thus, by (\ref{4.44}),
\beq\|w_t(2+w_t)\|_{\LL^1(S''_{h/2})}&&\leq\int_{S''_{h/2}}(|\E^2+2\E|+|W_t|^2)dgdt\nonumber\\
&&\leq C\|\vare(A)\|^2_{\LL^2(S''_{h/2})}+Ch^{1/2}\|\vare(A)\|_{\LL^2(S''_{h/2})}.\label{4.21}\eeq

Since $\t{S}$ is regular with negative curvature, we may assume that $\t{S}$ is given by
$$\t{S}=\{\,\a(s)\,|\,s=(s_1,s_2)\in(0,a)\times(0,b)\,\},$$
where $\a(\cdot):$ $(0,a)\times(0,b)\rw \t{S}$ is an immersion \cite{Yao2017,Yao2018}. We take
$$f_0(s)=2,\quad f(s)=w_t(\a(s)+tn(\a(s))\in\WW_0^{1,2}(\Xi),$$ in Lemma \ref{l4.3}. Applying (\ref{Xi}) to the above $f$ yields
$$\|w_t\|_{\LL^1(S'')}\leq C(\|w_t(2+w_t)\|_{\LL^1(S'')}+\|w_t(2+w_t)\|^{1/2}_{\LL^1(S'')}\|D w_t\|_{\LL^2(S'')}).$$
We integrate the above inequality with respect to $t$ over $(-h/2,h/2)$  to have
\be\|w_t\|_{\LL^1(S''_{h/2})}\leq C(\|w_t(2+w_t)\|_{\LL^1(S''_{h/2})}+\|w_t(2+w_t)\|^{1/2}_{\LL^1(S''_{h/2})}\|D w_t\|_{\LL^2(S''_{h/2})}).\label{Xi4.48}\ee It thus follows from (\ref{Xi4.48}),  (\ref{4.21}), and (\ref{4.27})  that
\beq\|w_t\|^2_{\LL^2(S''_{h/2})}&&\leq2\|w_t\|_{\LL^1(S''_{h/2})}+\|w_t(2+w_t)\|_{\LL^1(S''_{h/2})}\nonumber\\
&& \leq C\|w_t(2+w_t)\|_{\LL^1(S''_{h/2})}+C\|w_t(2+w_t)\|_{\LL^1(S''_{h/2})}^{1/2}\|\nabla^2(\psi y)\|_{\LL^2(S''_{h/2})}\nonumber\\
&&\leq C(\|\vare(\nabla u)\|^2_{\LL^2(S_{h})}+h^{1/2}\|\vare(\nabla u)\|_{\LL^2(S_h)})\nonumber\\
&&+C(\|\vare(\nabla u)\|^2_{\LL^2(S_{h})}
+h^{1/2}\|\vare(\nabla u)\|_{\LL^2(S_h)})^{1/2}h^{-1}\|\vare(\nabla u)\|_{\LL^2(S_h)}.\label{4.48}\eeq
Thus, by (\ref{4.38*}), (\ref{4.36}), (\ref{4.44}), (\ref{n3.44}), and (\ref{4.48}), we obtain
\beq\|\nabla(\psi y)\|^2_{\LL^2(S''_{h/2})}&&\leq C\|A\|^2_{\LL^2(S''_{h/2})}=C\|Dw+w\nabla\n\|^2_{\LL^2(S''_{h/2})}+C\|Dw-\nabla nW\|^2_{\LL^2(S''_{h/2})}\nonumber\\
&&\quad+C\|W_t\|^2_{\LL^2(S''_{h/2})}+C\|w_t\|_{\LL^2(S''_{h/2})}^2\nonumber\\
&&\leq  C(\|\vare(\nabla u)\|^2_{\LL^2(S_{h})}+h^{1/2}\|\vare(\nabla u)\|_{\LL^2(S_h)})+CP(u,h)\nonumber\\
&&\leq CP(u,h).\nonumber\eeq

(b)\,\,\,Let $S$ be the ruled surface (\ref{om}). We estimate $\|W_t\|^2_{\LL^2(S''_{h/2})}$ and $\|w_t\|^2_{\LL^2(S''_{h/2})},$ respectively.
Using (\ref{4.41}) and (\ref{4.42}), we have
\beq|W_t|^2&&=\sum_{i=1}^2[\<(\E'^2+2\E')E_i,E_i\>-\<(\E'^2+2\E')E_i,E_i\>]
+|Dw-\nabla nW|^2\nonumber\\
&&\leq C(|\E'|^2+|\E|^2+|\E'|+|\E|)+|Dw-\nabla nW|^2\qfq z=x+tn(x)\in S''_{h/2}.\nonumber\eeq
Integrating the above inequality with respect to $(x,t)$ over $S''\times(-h/2,h/2)$ yields
$$\|W_t\|^2_{\LL^2(S''_{h/2})}\leq C(\|\vare(\nabla u)\|^2_{\LL^2(S_h)}+h^{1/2}\|\vare(\nabla u)\|^2_{\LL^2(S_h)})+\|Dw-\nabla nW\|^2_{\LL^2(S''_{h/2})}.$$ We then obtain, by (\ref{n3.45}),
\be\|W_t\|^2_{\LL^2(S''_{h/2})}\leq CP(u,h),\label{n4.56}\ee
where $P(u,h)$ is given in (\ref{n3.46}).

By (\ref{om}), $w_t\Big(\om(\cdot)+tn(\om(\cdot))\Big)\in\WW^{1,2}_0((a_1,a_2)\times(b_1,b_2)),$ where $\om(s)=c(s_1)+s_2\delta(s_1).$ A similar argument as in (\ref{4.48}) yields
\be \|w_t\|^2_{\LL^2(S''_{h/2})}\leq CP(u,h).\label{n4.57}\ee

From (\ref{4.38*}), (\ref{4.36}), (\ref{n3.45}), (\ref{n4.56}), and (\ref{n4.57}), we obtain
\beq\|\nabla(\psi y)\|^2_{\LL^2(S''_{h/2})}&&\leq C\|DW+w\nabla n\|^2_{\LL^2(S''_{h/2})}+C\|Dw-\nabla nW\|^2_{\LL^2(S''_{h/2})}\nonumber\\
&&\quad+C\|W_t\|^2_{\LL^2(S''_{h/2})}+C\|w_t\|^2_{\LL^2(S''_{h/2})}\nonumber\\
&&\leq CP(u,h).\nonumber\eeq
\end{proof}

{\bf Proof of Theorem \ref{t1.1}.}\,\,\,Let $S$ be a parabolic. Then $S$ can be parameterized as a ruled surface. We assume that $\t{S}$ is a ruled surface, given by
$$\t{S}=\{\,c(s_1)+s_2\delta(s_1)\,|\,(s_1,s_2)\in(a_1,a_2)\times(b_1,b_2)\,\}$$ such that
$$ S\subset\subset S''\subset\subset\t{S}.$$

By \cite{GHa,Yao2021}, inequality (\ref{4.9*}) holds true with $\tau=3/2.$ Thus Theorem \ref{t4.2} holds for $\tau=3/2.$

Next, we divide the remaining proof into the following two cases:

{\bf Case 1.}\,\,Let $\b=15/8+\theta$ where $0<\theta<1/8$ is given small. We assume that $0<h<h_0$ and $u\in\H_0^1(S_h,\R^3)$ satisfy
$$\|\nabla(\psi y)\|^2_{\LL^2(S''_{h/2})}\geq h^{-\b}\|\vare(\nabla u)\|^2_{\LL^2(S_h)}.$$
By (\ref{n4.51}) in Lemma \ref{l4.10} and (\ref{n3.46}), we have
\beq h^{-\b}\|\vare(\nabla u)\|^2_{\LL^2(S_h)}&&\leq C(\|\vare(\nabla u)\|^2_{\LL^2(S_h)}+h^{1/2}\|\vare(\nabla u)\|_{\LL^2(S_h)})\nonumber\\
&&\quad+C(\|\vare(\nabla u)\|^2_{\LL^2(S_h)}+h^{1/2}\|\vare(\nabla u)\|_{\LL^2(S_h)})^{1/2}h^{-1}\|\vare(\nabla u)\|_{\LL^2(S_h)},\nonumber\eeq
from which we obtain
\be \|\vare(\nabla u)\|_{\LL^2(S_h)}\leq Ch^{1/2+\b}+Ch^{-3/4+\b}\|\vare(\nabla u)\|^{1/2}_{\LL^2(S_h)}.\label{n3.58}\ee

We claim that inequality (\ref{n3.58}) implies that
\be\|\vare(\nabla u)\|^2_{\LL^2(S_h)}\leq Ch^{3+\tau+\theta}\label{n3.59}\ee with $\tau=3/2.$
Indeed, if
$$h^{-3/4+\b}\|\vare(\nabla u)\|^{1/2}_{\LL^2(S_h)}\leq h^{1/2+\b},$$ then it follows from (\ref{n3.58}) that
$$\|\vare(\nabla u)\|^2_{\LL^2(S_h)}\leq Ch^{1+2\b}=Ch^{3+\tau+1/4+2\theta}\leq Ch^{3+\tau+\theta}.$$
Next, we assume
$$h^{-3/4+\b}\|\vare(\nabla u)\|^{1/2}_{\LL^2(S_h)}> h^{1/2+\b}.$$ From (\ref{n3.58}), we have
$$\|\vare(\nabla u)\|_{\LL^2(S_h)}\leq Ch^{-3/4+\b}\|\vare(\nabla u)\|^{1/2}_{\LL^2(S_h)},$$ which yields
$$\|\vare(\nabla u)\|^2_{\LL^2(S_h)}\leq Ch^{-3+4\b}=Ch^{3+\tau+4\theta}\leq Ch^{3+\tau+\theta}.$$

From (\ref{n3.59}) and by Theorem \ref{t4.2}, we obatin
$$\|\nabla u\|^2_{\LL^2(S_h)}\leq Ch^{-3/2}\|\vare(\nabla u)\|^2_{\LL^2(S_h)} \leq Ch^{-15/8-\theta}\|\vare(\nabla u)\|^2_{\LL^2(S_h)}.    $$

{\bf Case 2.}\,\,\,Suppose
$$\|\nabla(\psi y)\|^2_{\LL^2(S''_{h/2})}\leq h^{-\b}\|\vare(\nabla u)\|^2_{\LL^2(S_h)}$$ with $\b=15/8+\theta.$
By Lemma \ref{l4.1} and (\ref{3.9}), we have
$$\|\nabla y\|^2_{\LL^2(S_h)}\leq C\|\nabla(\psi y)\|^2_{\LL^2(S''_{h/2})}+C\|\dist^2(\cdot,\pl(\t{S}_h))\nabla^2 y\|^2_{\LL^2(\t{S}_h)}\leq Ch^{-\b}\|\vare(\nabla u)\|^2_{\LL^2(S_h)}.$$ It thus follows from Lemma \ref{l4.6} that
\beq\|\nabla u\|^2_{\LL^2(S_h)}&&\leq C\|\vare(\nabla u)\|^2_{\LL^2(S_h)}+\|\nabla y\|^2_{\LL^2(\t{S}_h)}\nonumber\\
&&\leq C\|\vare(\nabla u)\|^2_{\LL^2(S_h)}+\|\nabla y\|^2_{\LL^2(S_h)}+\|\nabla y\|^2_{\LL^2((\t{S}\backslash S)_h)}\nonumber\\
&&\leq Ch^{-15/8-\theta}\|\vare(\nabla u)\|^2_{\LL^2(S_h)}.\label{4.60}\eeq
\hfill$\Box$

{\bf Proof of Theorem \ref{t1.2}.}\,\,\,
By \cite{Har,Yao2018} inequality (\ref{4.9*}) holds true with $\tau=4/3,$ that is, Theorem \ref{t4.2} holds for $\tau=4/3.$

{\bf Case 1.}\,\,\,Suppose $0<h<h_0$ and $u\in\H_0^1(S_h,\R^3)$ are such that
$$\|\nabla(\psi y)\|^2_{\LL^2(S''_{h/2})}\geq h^{-\b}\|\vare(\nabla u)\|^2_{\LL^2(S_h)}$$ where $\b=11/6+\theta.$ It follows from (\ref{n4.51}) that
\beq h^{-\b}\|\vare(\nabla u)\|^2_{\LL^2(S_h)}&&\leq C(\|\vare(\nabla u)\|^2_{\LL^2(S_h)}+h^{1/2}\|\vare(\nabla u)\|_{\LL^2(S_h)})\nonumber\\
&&\quad+C(\|\vare(\nabla u)\|^2_{\LL^2(S_h)}+h^{1/2}\|\vare(\nabla u)\|_{\LL^2(S_h)})^{1/2}h^{-1}\|\vare(\nabla u)\|_{\LL^2(S_h)},\nonumber\eeq which yields
\be\|\vare(\nabla u)\|_{\LL^2(S_h)}\leq Ch^{14/6+\theta}+Ch^{13/12+\theta}\|\vare(\nabla u)\|^{1/2}_{\LL^2(S_h)}.\label{n4.60}\ee From (\ref{n4.60}) we have
$$   \|\vare(\nabla u)\|^2_{\LL^2(S_h)}\leq Ch^{13/3+\theta}.     $$
By Theorem \ref{t4.2} with $\tau=4/3,$ we obtain
$$\|\nabla u\|^2_{\LL^2(S_h)}\leq Ch^{-4/3}\|\vare(\nabla u)\|^2_{\LL^2(S_h)}.$$

{\bf Case 2.}\,\,\,Let
$$\|\nabla(\psi y)\|^2_{\LL^2(S''_{h/2})}< h^{-11/6-\theta}\|\vare(\nabla u)\|^2_{\LL^2(S_h)}.$$ A similar argument as in (\ref{4.60}) gives
$$\|\nabla u\|^2_{\LL^2(S_h)}\leq Ch^{-11/6-\theta}\|\vare(\nabla u)\|^2_{\LL^2(S_h)}.$$ \hfill$\Box$

{\bf Proof of Theorem \ref{t1.4}.}\,\,\,(i)\,\,\,Let $0<\mu<2$ be such that $(\ref{1.1})$ holds true. Suppose that $w\in\WW^{2,2}(S,\R^3)$ is an isometry from $S$ to $\R^3.$ Let $\bar{n}(x)$ be the normal field on surface $w(S)$ at $w(x)$ with positive orientation induced from the normal $n$ of $S.$

We define
$$u(z)=w(x)+t\bar{n}(x)\qfq z=x+tn(x)\in S_h.$$ Then $u\in\WW^{1,2}(S_h,\R^3).$ Fix $x\in S.$ Let $\{E_1,E_2\}$ be an orthonormal basis of $S_x$ with positive orientation. Then
$$\div(E_1,E_2,n(x))=1,\quad \bar{n}(x)=\nabla_{E_1}w\times\nabla_{E_2}w.$$
Denote
$$Q=(E_1,E_2,n),\quad R(w)=(\nabla_{E_1}w,\nabla_{E_2}w,\bar{n}).$$ Then $Q,$ $R(w)\in\SO(3).$ At point $x$ we have
$$\nabla u(z)n=\bar{n},\quad \nabla u(z)(E_i+t\nabla nE_i)=\nabla_{E_i}w+t\nabla_{w_*E_i}\bar{n}\qfq i=1,\,\,2.$$ It follows  from $\nabla nn=0$ and the above formulas that
\be\nabla uQ=(\nabla_{E_1}u,\nabla_{E_2}u,\nabla_nu)=R(w)+t(B-\nabla u\nabla nQ),\label{4.51}\ee where
$$B=(\nabla_{w_*E_1}\bar{n},\nabla_{w_*E_1}\bar{n},0).$$

From (\ref{4.51}) we obtain
$$Q^T[\nabla u-R(w)Q^T]Q=t Q^T(B-\nabla u\nabla nQ). $$ Since $R(w)Q^T\in\SO(3),$
\beq\dist^2(\nabla u,\SO(3))&&\leq|\nabla u-R(w)Q^T|^2=t^2|B-\nabla u\nabla nQ|^2\nonumber\\
&&\leq Ch^2(|\nabla^2w|^2+|\nabla u|^2)\qaq x. \label{4.52}\eeq

Let $\{h_k\}$ be a sequence of positive numbers with
$$h_k\rw0\qasq k\rw\infty.$$ By inequalities (\ref{1.1}) and (\ref{4.52}) there are $R_k\in\SO(3)$ such that
\beq\|\nabla u-R_k\|^2_{\LL^2(S_{h_k},\R^{3\times3})}&&\leq\frac{C}{h^\mu}\|\dist(\nabla u,\SO(3))\|^2_{\LL^2(S_{h_k})}
\nonumber\\
&&\leq Ch_k^{2-\mu}(h_k\|D^2w\|^2_{\LL^2(S,\R^{2\times2})}+\|\nabla u-R_k\|^2_{\LL^2(S_{h_k},\R^{3\times3})}+3h_k|S|),\nonumber\eeq which yields
\be \|\nabla u-R_k\|^2_{\LL^2(S_{h_k},\R^{3\times3})}\leq Ch_k^{3-\mu}(\|D^2w\|^2_{\LL^2(S,\R^{2\times2})}+|S|)\label{4.53}\ee for $k$ large enough since $0<\mu<2.$ We may assume that for some $R_0\in\SO(3)$
$$R_k\rw R_0\qasq k\rw\infty.$$

Moreover, form (\ref{4.51}) we have
\beq|\nabla u-R_k|^2&&=|(\nabla u-R_k)Q|^2=|R(w)-R_kQ|^2+2t\<R(w)-R_kQ,B-\nabla u\nabla nQ\>\nonumber\\
&&\quad+t^2|B-\nabla u\nabla nQ|^2\nonumber\\
&&\geq\frac12|R(w)-R_kQ|^2-6t^2(|B|^2+|(\nabla u-R_k)\nabla nQ|^2+|R_k\nabla nQ|^2),\nonumber\eeq that is,
\be|R(w)-R_kQ|^2\leq C|\nabla u-R_k|^2+Ch_k^2(|B|^2+1)\label{4.54}\ee for $k$ large enough.

Finally, using (\ref{4.53}) and (\ref{4.54}) we obtain
$$R(w)=R_0Q\qaq x.$$ Thus there is some $a\in\R^3$ such that $w(x)=R_0x+a$ for $x\in S.$\\

(ii)\,\,\,Let $0<\mu<2$ be such that $(\ref{1.2})$ holds true. A similar argument as in (i) completes the proof.
\hfill$\Box$

{\bf Compliance with Ethical Standards}

Conflict of Interest: The author declares that there is no conflict of interest.

Ethical approval: This article does not contain any studies with human participants or animals performed by the authors.

 \end{document}